\documentclass[11pt]{amsart}
\usepackage{preamble}

\author{Noam Kimmel}
\address{N. Kimmel: Raymond and Beverly Sackler School of Mathematical Sciences, Tel Aviv University, Tel
Aviv 69978, Israel.}
\email{\href{mailto:noamkimmel@mail.tau.ac.il}{noamkimmel@mail.tau.ac.il}}
\title{Asymptotic zeros of Poincar\'e series}
\thanks{This research was supported by the European Research Council (ERC) under the European Union's  Horizon 2020 research and innovation program  (Grant agreement No.    786758).}

\begin{document}

\maketitle

\begin{abstract}
We study the zeros of Poincar\'e series $P_{k,m}$ for the full modular group.
We consider the case where $m \sim \alpha k$ for some constant $\alpha > 0$. 
We show that in this case a positive proportion of the zeros lie on the line $\frac{1}{2} + it$.
We further show that if $\alpha > \frac{\log(2)}{2\pi}$ then the imaginary axis also contains a positive proportion of zeros.
We also give a description for the location of the non-real zeros when $\alpha$ is small. 
\end{abstract}

\setcounter{tocdepth}{1}
\tableofcontents

\section{Introduction}
\subsection{Zeros of modular forms}
For an even integer $k\geq 0$, we denote by $M_k$ the space of modular forms of weight $k$ for the full modular group $\text{SL}(2,\ZZ)$.
The Fourier expansion of $f\in M_k$ is given by
$$
f(z) = \sum_{n\geq 0}a_f(n) q^n
$$
with $q = e^{2\pi i z}$, $z\in \HH = \SET{z\;:\; \Im{z} > 0}$.
One defines $v_\infty(f)$ as the smallest $n$ such that $a_f(n) \neq 0$.
Forms with $v_\infty(f) > 0$ form a subspace of $M_k$ called the space of cusp forms, and denoted by $S_k$.
The space $M_k$ is the direct sum of $S_k$ and the one dimensional space spanned by the Eisenstein series of weight $k$:
$$
E_k(z) = \frac{1}{2}\sum_{\substack{c,d\in \ZZ \\ \gcd(c,d) = 1}}\frac{1}{(cz + d)^k}.
$$
Writing $k = 12 \ell + k'$ with $k'\in\SET{0,4,6,8,10,14}$, it is known that $M_k$ is of dimension $\ell + 1$.

A form $f\in M_k$ has $\frac{k}{12} + \BigO{1}$ zeros in the fundamental domain
\begin{multline*} 
\mathcal{F} = 
\SET{z\in \HH \; : \; |z| > 1, \Re{z} \in \left(-\frac{1}{2},\frac{1}{2}\right]} \\
\bigcup 
\SET{z\in \HH \; : \; |z| = 1, \arg{z} \in \left[\frac{\pi }{3}, \frac{\pi}{2}\right]}
\end{multline*}
(see \autoref{fig-fund_domain}).
More precisely, a nonzero form $f\in M_k$ satisfies the valence formula:
\begin{equation*}
v_\infty(f) + \frac{1}{2}v_i(f) + \frac{1}{3}v_\rho(f) + 
\sum_{\substack{p\in \mathcal{F} \\ p\neq i,\rho}} v_p(f)
= \frac{k}{12}
\end{equation*}
where $\rho = \frac{1}{2} + \frac{\sqrt{3}}{2}i$, and $v_p(f)$ is the order of vanishing of $f$ at $p$.

We define the geodesic segments
\begin{equation*}
\mathcal{A} = \SET{e^{i\theta} \;\middle|\; \theta \in \left[\frac{\pi}{3},\frac{\pi}{2}\right]} , \;
\mathcal{L}_\rho = \SET{\frac{1}{2} + ti \;\middle| \; t\geq \frac{\sqrt{3}}{2}} ,\;
\mathcal{L}_i = \SET{ti \;\middle| \; t\geq 1},
\end{equation*}
as are shown in \autoref{fig-fund_domain}.
\begin{figure}[ht]
\centering
\includegraphics[trim=0 3cm 0 3cm, clip, width=1\textwidth]{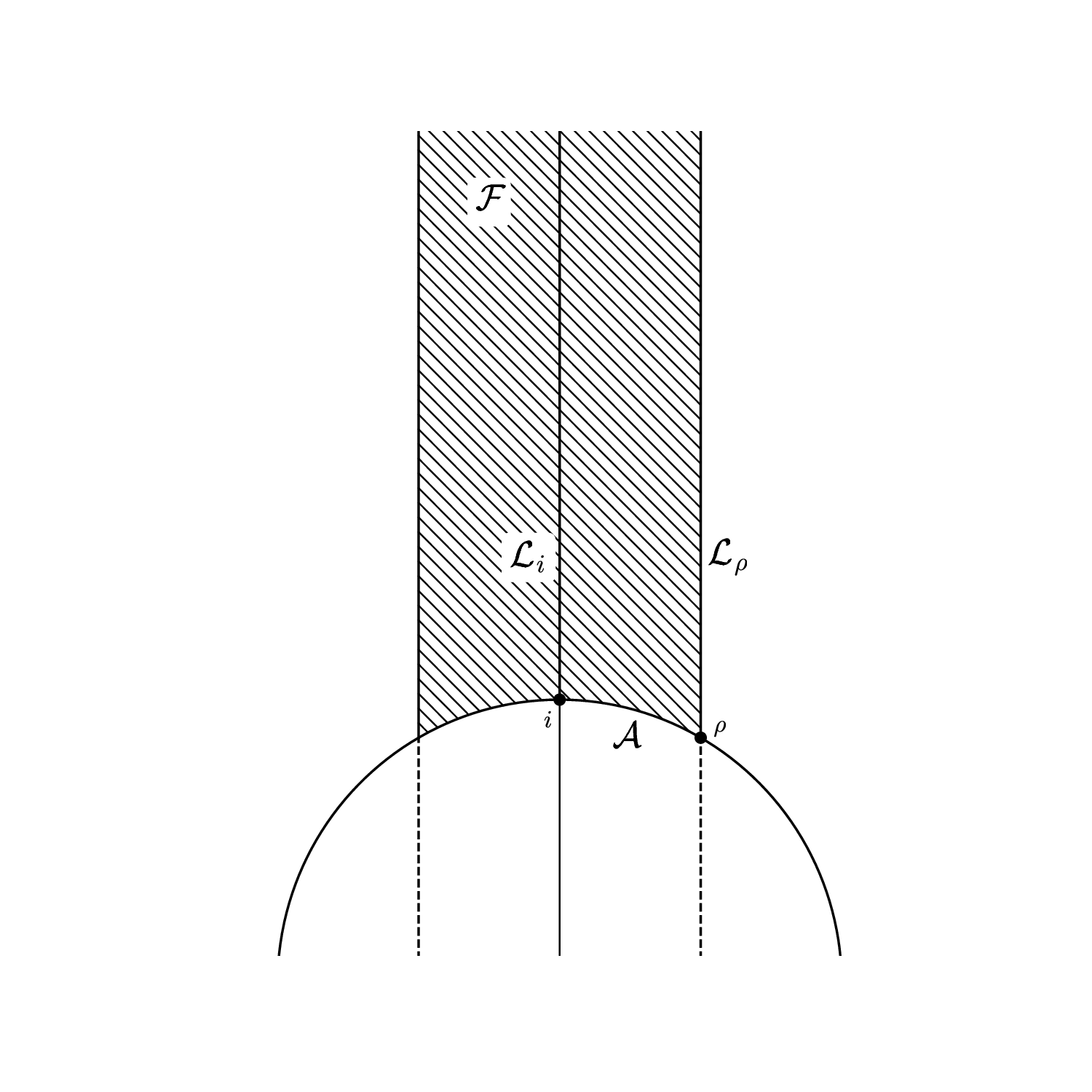}
\caption{The fundamental domain $\mathcal{F}$, and the three segments $\mathcal{L}_i$, $\mathcal{L}_\rho$, and $\mathcal{A}$.}\label{fig-fund_domain}
\end{figure}
Following the terminology of \cite{MR2881302}, we call zeros of a form $f\in M_k$ lying on these segments 'real' zeros.
Note that the $j$-invariant takes real values on these segments.
Also, if the $a_f(n)$'s are all real, then $f(z)$ takes real values on $\mathcal{L}_\rho, \mathcal{L}_i$, and $z^{k/2}f(z)$ takes real values on $\mathcal{A}$.

There have been numerous results regarding the distribution of zeros, and of real zeros, for various modular forms of interest.
If $f\in M_k$ is a Hecke eigenform, a consequences of the holomorphic QUE conjecture proved by Holowinsky and Soundararajan \cite{MR2680499} is that the zeros of $f$ equidistribute in $\mathcal{F}$ with respect to the hyperbolic measure as $k\rightarrow\infty$.
This consequence was proved by Rudnick in \cite{MR2181743}.
It was also proved by Ghosh and Sarnak in \cite{MR2881302} that the number of real zeros of such forms tends to infinity with $k$.

The situation for the Eisenstein series is quite different. 
It was shown by Rankin and Swinnerton-Dyer \cite{MR0664646} that all zeros of $E_k$ in $\mathcal{F}$ lie on the arc $\mathcal{A}$.
There have been many generalizations of this result to other modular forms and functions of interest.
For example, in \cite{MR2441704} Duke and Jenkins extended this result to certain types of "gap forms".
In \cite{MR2302560} Hahn considered Eisenstein series associated with other Fuchsian groups. 
In \cite{vanittersum2023zeros} Ittersum and Ringeling considered odd weight Eisenstein series.
In \cite{MR3706602} and \cite{zbMATH07399193} zeros of certain combinations of Eisenstein series were considered.
In \cite{MR2265867} Gun considered the zeros of certain linear combinations of Poincar\'e series.
In \cite{MR2409175} Nozaki proved an interlacing property between the zeros of $E_{k+12}$ and those of $E_k$, which was expanded on in \cite{MR4323910}.
This interlacing property was also studied for other modular forms of interest in \cite{MR3211795}, \cite{MR4613381}, \cite{MR2929008}.

\subsection{Main results}
In this paper we consider the zeros of Poincar\'e series.
Denote by $P_{k,m}(z)$ the Poincar\'e series 
$$
P_{k,m}(z) = \frac{1}{2}
\sum_{\substack{c,d\in\ZZ \\ \gcd(c,d) = 1}}
\frac{\exp{2\pi i m \gamma z}}{(cz + d)^k} \in M_k
$$
where $\gamma \in \text{SL}(2,\ZZ)$ is any matrix of the form 
$
\gamma = 
\begin{pmatrix}
    a       & b  \\
    c       & d 
\end{pmatrix}.
$
For $m = 0$ we have that $P_{k,0} = E_k$ is the Eisenstein series.
For $m\geq 1$ it is known that $P_{k,m}$ is a cusp form.
It is also known that $P_{k,m}$ with $0\leq m \leq \ell$ span $M_k$.

In \cite{MR0664646} Rankin showed that $P_{k,m}$ has at least $\frac{k}{12} - m$ zeros on $\mathcal{A}$.
So if $m=o(k)$, this proves that $100\%$ of the zeros of $P_{k,m}$ lie on $\mathcal{A}$ as $k$ goes to infinity.
However, for larger values of $m$, this still leaves a significant proportion of zeros unaccounted for, and for $m\geq \frac{k}{12}$ this does not give any information regarding the zeros of $P_{k,m}$.

\begin{figure}[ht]
\centering
\includegraphics[trim=4cm 2cm 4cm 0cm, clip, width=1\textwidth]{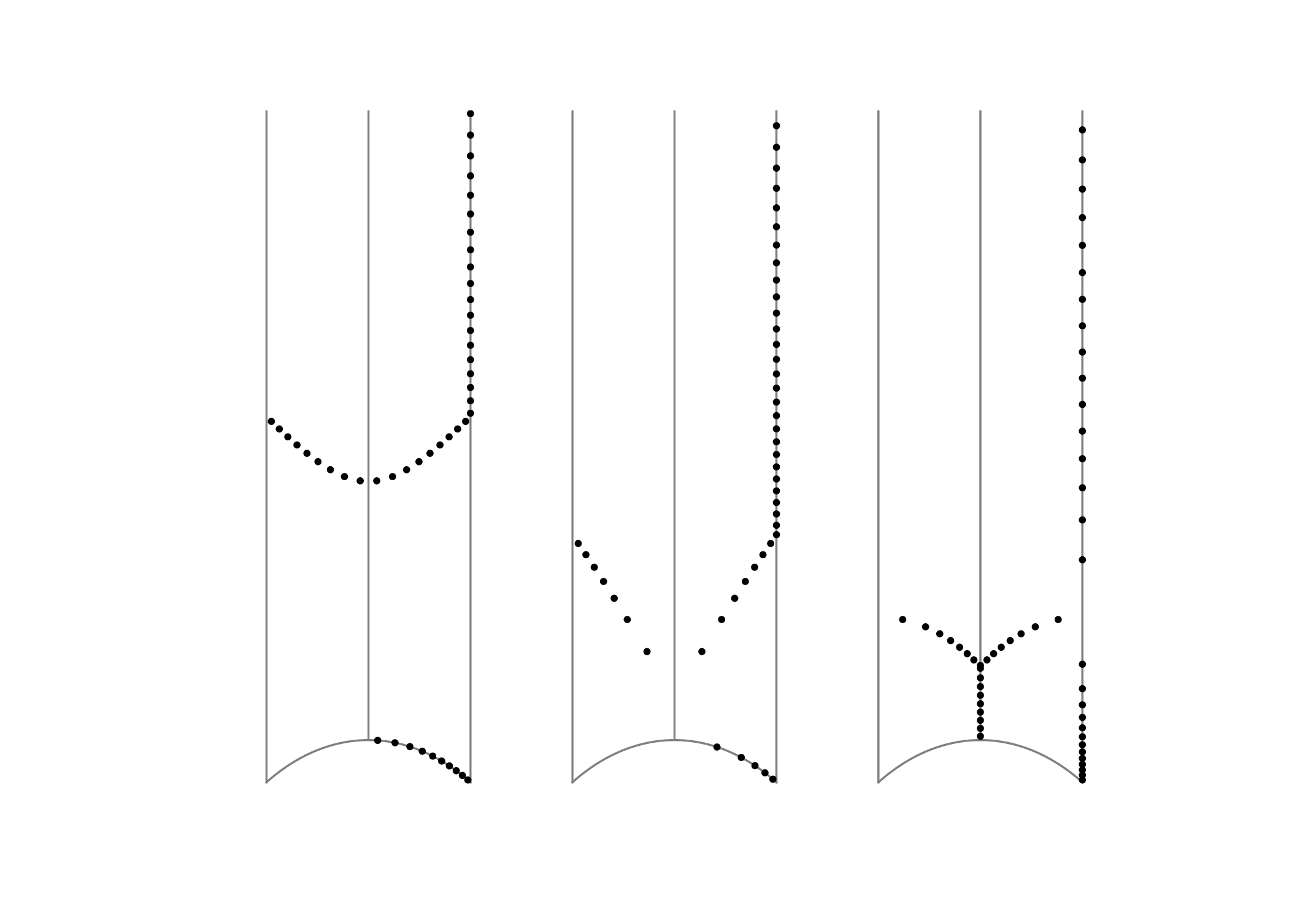}
\caption{From left to right, the zeros in $\mathcal{F}$ of $P_{1200,90}$, $P_{1200,95}$, $P_{1200,150}$ with $\Im{ z} < 3$.}
\label{fig-zeros}
\end{figure}

In this paper, we explore the zeros of $P_{k,m}$ where $m = \alpha k + o(k)$ for some constant $\alpha>0$, as $k$ goes to infinity.
In \autoref{fig-zeros} we show a plot of the zeros of $P_{1200,m}$ for various $m$'s.
In these images we see that there are many zeros on $\mathcal{L}_\rho$, and in some cases many lie on $\mathcal{L}_i$ as well.
The non-real zeros also seem to follow a very neat pattern.
We aim to explain these phenomena in this paper.
Specifically, we prove the following theorems.

\begin{theorem}\label{thm-lrho}
Assume $m = \alpha k + o(k)$ for some $\alpha > 0$.
Then a positive proportion of the zeros of $P_{k,m}$ in $\mathcal{F}$ lie on $\mathcal{L}_\rho$.
That is, there exists $P_\rho(\alpha) > 0$ such that the number of zeros of $P_{k,m}$ on $\mathcal{L}_\rho$ is at least $P_\rho(\alpha) \frac{k}{12} + o(k)$.
\end{theorem}
\begin{remark}
For $\alpha \leq \frac{\sqrt{3}\log(3)}{4\pi}$ we give an explicit formula for a constant $P_\rho(\alpha)$.
That is, for $\alpha \leq \frac{\sqrt{3}\log(3)}{4\pi}$ we show that we can set
\begin{equation*}
P_\rho(\alpha) =  \begin{cases}
    12\left(\frac{1}{2} - \frac{ \alpha}{\frac{1}{4} + F^{-1}(\alpha)^2} - \frac{1}{\pi}\arctan{2F^{-1}(\alpha)}\right) & \alpha < \frac{1}{2\pi\sqrt{3}} \\
    12\left(\frac{5}{6} + \alpha
    - \frac{2\alpha}{\frac{1}{4} + t_\text{min}^2}
    - \frac{2}{\pi}\arctan{2t_\text{min}}\right)
     & \frac{1}{2\pi\sqrt{3}} \leq \alpha \leq \frac{\sqrt{3}\log(3)}{4\pi}
\end{cases}
\end{equation*}
where 
$$
F(t) = 
\frac{1}{4\pi}
\frac{\qpt\log\qpt}{t\left(t^2 - \frac{3}{4}\right)}
$$
and 
$$
t_\text{min} = 2\pi\alpha + \sqrt{4\pi^2\alpha^2 - \frac{1}{4}}.
$$
See \autoref{fig-palpha} for a plot of these $P_\rho(\alpha)$.

In this range of $0<\alpha \leq \frac{\sqrt{3}\log(3)}{4\pi}$, we believe the constants $P_\rho(\alpha)$ we attain are sharp.
For $\alpha < \frac{1}{4\pi}$ we prove that this is in fact true (this is shown in the proof of \autoref{thm-non_real}).
That is, for $\alpha < \frac{1}{4\pi}$ we show that the number of zeros of $P_{k,m}$ that lie on $\mathcal{L}_\rho$ is equal $P_{\rho}(\alpha) \frac{k}{12} + o(k)$ for the constant $P_\rho(\alpha)$ above.
For $\alpha > \frac{\sqrt{3}\log(3)}{4\pi}$ we only show $P_\rho(\alpha) > 0$ which simplifies calculations.
\end{remark}

\begin{theorem}\label{thm-li}
Assume $m = \alpha k + o(k)$ for some $\alpha > \frac{\log(2)}{2\pi}$.
Then a positive proportion of the zeros of $P_{k,m}$ in $\mathcal{F}$ lie on $\mathcal{L}_i$.
That is, there exists $P_i(\alpha) > 0$ such that the number of zeros of $P_{k,m}$ on $\mathcal{L}_i$ is at least $P_i(\alpha) \frac{k}{12} + o(k)$.
\end{theorem}

\begin{theorem}\label{thm-non_real}
Assume $m = \alpha k + o(k)$ for some $0<\alpha < \frac{1}{4\pi}$.
Let $\Gamma_\alpha\subset \mathcal{F}$ be the curve defined by the equation 
$$
2\pi\alpha y = \frac{\log(|z|)}{1 - |z|^{-2}}
$$
where $z = x + iy$, $|z| > 1$ (see \autoref{fig-gamma_alpha}).
Let $\varepsilon>0$, and denote by $N_{\Gamma_\alpha, \varepsilon}$ the number of non-real zeros of $P_{k,m}$ in $\mathcal{F}$ whose distance from $\Gamma_\alpha$ is at most $\varepsilon$.
Then we have $N_{\Gamma_\alpha, \varepsilon} > C(\alpha) k + o(k)$ for some positive constant $C(\alpha)$.
Furthermore, these zeros account for $100\%$ of the zeros not on $\mathcal{A}$, $\mathcal{L}_\rho$.
That is, all but at most $o(k)$ of the zeros of $P_{k,m}$ in $\mathcal{F}$ not on $\mathcal{A}$, $\mathcal{L}_\rho$ are non-real zeros whose distance to $\Gamma_\alpha$ is at most $\varepsilon$.
\end{theorem}

\begin{figure}[ht]
\centering
\includegraphics[trim=0 3cm 0 3cm , clip, width=1\textwidth]{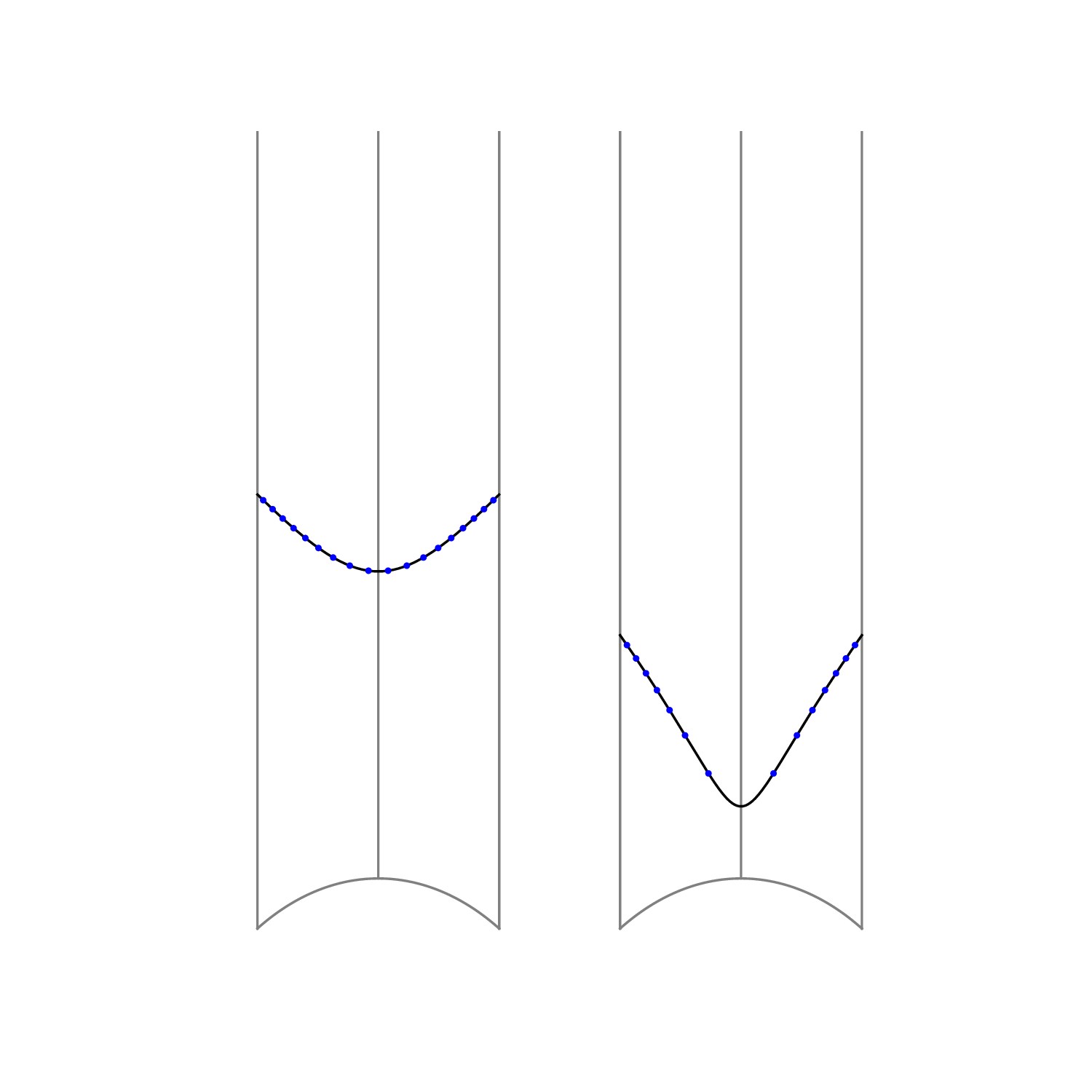}
\caption{The curve $\Gamma_\alpha$ for $\alpha = \frac{90}{1200}$ (left) and $\alpha = \frac{95}{1200}$ (right), and the non-real zeroes of $P_{1200,90}$ and $P_{1200,95}$.}
\label{fig-gamma_alpha}
\end{figure}

\begin{remark}
The zeros in \autoref{fig-gamma_alpha} are not exactly on the curve $\Gamma_\alpha$, only very close to it.
In the following table we calculate for each non-real zero $x+iy$ of $P_{1200,90}$ the value $2\pi\alpha y - \log(|z|)\left(1 - |z|^{-2}\right)^{-1}$ for $\alpha = \frac{90}{1200}$:

\text{ }

\centering
\begin{tabular}{c|c}
    non-real zeros of $P_{1200,90}$ & $2\pi\alpha y - \log(|z|)\left(1 - |z|^{-2}\right)^{-1}$\\
    \hline \hline
    $\approx 0.48 + 2.01 i$ & $\approx10^{-4.66}$ \\
    \hline
    $\approx 0.44 + 1.99 i$ & $\approx10^{-7.58}$ \\
    \hline
    $\approx 0.40 + 1.96 i$ & $\approx10^{-12.00}$ \\
    \hline
    $\approx 0.35 + 1.94 i$ & $\approx10^{-13.88}$ \\
    \hline
    $\approx 0.30 + 1.91 i$ & $\approx10^{-15.80}$ \\
    \hline
    $\approx 0.25 + 1.88 i$ & $\approx10^{-16.24}$ \\
    \hline
    $\approx 0.19 + 1.86 i$ & $\approx10^{-16.81}$ \\
    \hline
    $\approx 0.12 + 1.84 i$ & $\approx10^{-16.22}$ \\
    \hline
    $\approx 0.04 + 1.82 i$ & $\approx10^{-16.52}$ 
\end{tabular}

\end{remark}

 % introduction

\section{Notations and preliminaries}\label{sec-2}
\subsection{Notations}
Let $\alpha > 0$. 
All asymptotic notations are to be understood in the limit $k\rightarrow\infty$.
We allow all implied constants to depend on $\alpha$.
In several places in the paper we will fix some compact subset $\mathcal{C}\subset \mathcal{F}$, or some closed interval $[A,B]\subset \RR$.
In these situations, we allow the implied constants to depend on $\mathcal{C}$, $[A,B]$ as well.
Any other dependence will be specified.

For given $k,m$, we denote
\begin{equation}\label{def-sigmacd}
\sigma_{c,d}(z)
=
\frac{\exp{2\pi i m \gamma z}}{(cz + d)^k}
\end{equation}
so that 
\begin{equation*} 
P_{k,m}(z) = \frac{1}{2}
\sum_{\substack{c,d\in\ZZ \\ \gcd(c,d) = 1}}
\sigma_{c,d}(z).
\end{equation*}
Since $k$ is even, $\sigma_{c,d} = \sigma_{-c,-d}$ (if $a,b$ correspond to $c,d$ then $-a,-b$ correspond to $-c,-d$).
And so we have that
\begin{equation} \label{def-pkm}
P_{k,m}(z) = \sigma_{0,1}(z) + \sum_{\substack{c,d\in\ZZ \\ c>0 \\ \gcd(c,d) = 1}}\sigma_{c,d}(z).
\end{equation}
From here on out, we assume all pairs $(c,d)$ to satisfy $c > 0$, $\gcd(c,d) = 1 $ or $(c,d) = (0,1)$ even when not explicitly written.

Denoting $z = x + iy$, we define
\begin{equation} \label{def-rcd}
r_{c,d}(z) = \frac{\exp{-2\pi\alpha \frac{y}{\left| cz + d\right|^2}}}{\left|cz + d\right|}.
\end{equation}
We see that
$$
\left|\sigma_{c,d}(z)\right| = 
r_{c,d}(z)^{k }
\left(\exp{-2\pi\frac{y}{|cz+d|^2}}\right)^{k\alpha - m}.
$$
Thus, on any compact subset $\mathcal{C}\subset\mathcal{F}$ we have that
$$
r_{c,d}(z)^{k } C_1^{o(k)} \ll
\left|\sigma_{c,d}(z)\right| \ll
r_{c,d}(z)^{k } C_2^{o(k)}
$$
for some constants $C_1,C_2$ depending on $\mathcal{C}$ (but not on $(c,d)$).
We conclude that on $\mathcal{C}$:
\begin{equation}\label{eq-rcd^k}
\left|\sigma_{c,d}(z)\right| = 
\left(r_{c,d}(z)(1 + o(1))\right)^k
\end{equation}
uniformly in $(c,d),z$.

In \autoref{sec-5} we will also consider the logarithm of complex numbers in $\HH$.
In this case we assume the principal branch of the logarithm function:
$$
\log : \HH \rightarrow \SET{x + iy \, : \, y\in(0,\pi)}.
$$
We will also use the principal argument function $\text{Arg} : \HH \rightarrow (0,\pi)$.

We define 
$$
w_{c,d}(z) = \frac{y}{|cz + d|^2}
$$
and 
$$
I(w) = -2\pi \alpha w + \frac{1}{2}\log(w).
$$
With these notations we have
\begin{equation}\label{def-Iw}
\log\left(r_{c,d}(z)\right) = I(w_{c,d}(z)) -\frac{1}{2}\log(y).
\end{equation}

\subsection{Main lemma}
As noted in \eqref{eq-rcd^k}, we have that on any compact subset $\mathcal{C}$ of $\mathcal{F}$:
$$
\left|\sigma_{c,d}(z)\right| = \left(r_{c,d}(z)(1 + o(1))\right)^{k }
$$
and the term $r_{c,d}(z)$ does not depend on $k$.
We now want to show that sums of the form $\sum \sigma_{c,d}$ are dominated by those $\sigma_{c,d}$'s for which $r_{c,d}$ is maximal.

\begin{lemma}\label{lem-main}
Let $\mathcal{C}\subset \mathcal{F}$ be a compact subset.
Let $\mathcal{I}$ be a subset of the indices $(c,d)$.
Assume that we have for some $(c_m,d_m)\not\in\mathcal{I}$:
$$
r_{c_m,d_m}(z) > r_{c,d}(z)
$$
for all $(c,d)\in\mathcal{I}$ and for all $z\in\mathcal{C}$.
Then
$$
\sum_{(c,d)\in\mathcal{I}}\left| \sigma_{c,d}(z)\right|
= o\left(\left|\sigma_{c_m,d_m}(z)\right|\right)
$$
uniformly for $z\in\mathcal{C}$.
\end{lemma}

\begin{proof}
We first note that
\begin{equation*}
r_{c,d}(z) \leq \frac{1}{\left|cz + d\right|},\quad
\left|\sigma_{c,d}(z)\right| \leq \frac{1}{\left|cz + d\right|^k}.
\end{equation*}
We denote
$$
\nu
=
\min_{z\in\mathcal{C}}r_{c_m,d_m}(z).
$$
We claim that $\nu<1$.
This follows from the fact that $|c_mz + d_m| \geq 1$ for $z\in \mathcal{F}$, so that we have
$$
r_{c_m,d_m}(z) = \frac{\exp{-2\pi\alpha \frac{y}{\left| c_mz + d_m\right|^2}}}{\left|c_mz + d_m\right|}
\leq \exp{-2\pi\alpha \frac{y}{\left| c_mz + d_m\right|^2}}.
$$
This last expression is strictly smaller than $1$ since $\alpha > 0$ and $\frac{y}{\left| c_mz + d_m\right|^2}>0$, and so we have $\nu<1$.

We denote also 
$$
\delta =
\max_{\substack{z\in\mathcal{C}\\ (c,d) \in \mathcal{I}}}
\frac{r_{c,d}(z)}{r_{c_m,d_m}(z)}.
$$
This maximum is achieved since $\mathcal{C}$ is compact, and since $r_{c,d}(z) \xrightarrow{c,d\rightarrow\infty} 0 $.
Note also that $\delta < 1$ from our assumption that $r_{c_m,d_m}$ is maximal in $\mathcal{C}$.

Let $N > 0$ be such that
$$
\sum_{\substack{(c,d)\in\mathcal{I} \\ |cz + d| > N}}
\frac{1}{|cz + d|^3} < \nu^6
$$
for all $z\in\mathcal{C}$.
Denote by $\mathcal{N}$ the finite set of all $(c,d)\in\mathcal{I}$ such that 
$\min_{z\in\mathcal{C}} |cz + d| < N$.

For $(c,d)$'s in $\mathcal{N}$ we have that
\begin{multline*}
\frac{1}{\left|\sigma_{c_m,d_m}(z) \right|}
\sum_{\substack{(c,d)\in \mathcal{N}}}\left|\sigma_{c,d}(z) \right|
= \\
\sum_{\substack{(c,d)\in \mathcal{N}}}
\left(\frac{r_{c,d}(z)(1 + o(1))}{r_{c_m,d_m}(z)(1 + o(1))}\right)^{k}
< 
\left|\mathcal{N}\right|\cdot\delta^k (1 + o(1))^k.
\end{multline*}

We now consider those $(c,d)$'s in $\mathcal{I}$ which are not in $\mathcal{N}$.
We get that
\begin{multline*}
\frac{1}{\left|\sigma_{c_m,d_m}(z) \right|}
\sum_{\substack{(c,d)\in\mathcal{I} \\ (c,d)\not\in \mathcal{N}}}
\left|\sigma_{c,d}(z) \right|
< \\
\left(\nu(1 + o(1))\right)^{-k}
\sum_{\substack{(c,d)\in\mathcal{I} \\ |cz + d| > N}}
\frac{1}{|cz + d|^{k}} 
\leq \\
\left(\nu(1 + o(1))\right)^{-k}
\left(\sum_{\substack{(c,d)\in\mathcal{I} \\ |cz + d| > N}}
\frac{1}{|cz + d|^3}\right)^{k/3}
< \\
\left(\nu(1 + o(1))\right)^{-k} \left(\nu^6\right)^{k/3}
<
\nu^{k} (1 + o(1))^{-k}.
\end{multline*}
For the second inequality we used the monotonicity of norms $\|\cdot\|_k \leq \|\cdot\|_3$.

Overall, we have shown that
$$
\frac{1}{\left|\sigma_{c_m,d_m}(z) \right|}
\sum_{\substack{(c,d)\in \mathcal{I}}}\left|\sigma_{c,d}(z) \right|
<
\left|\mathcal{N}\right|\cdot \delta^{k}(1 + o(1))^k + \nu^{k} (1 + o(1))^{-k}
$$
which is independent of $z$ and tends to $0$ as $k\rightarrow\infty$ since $\delta,\nu < 1$.
\end{proof}
 % notations and preliminaries

\section{Zeros on \texorpdfstring{$\mathcal{L}_\rho$}{the line 0.5 + it}}\label{sec-3}

\subsection{Sketch of proof}

The Fourier coefficients of $P_{k,m}$ are known to be real. 
Thus, as previously stated, it is known that $P_{k,m}\hpt$ is real for $t\geq \mint$.
This can also be seen directly from the definition \eqref{def-pkm} since 
$$
\sigma_{c,d}\hpt = \overline{\sigma_{c,-c-d}\hpt},
$$
which follows from \eqref{def-sigmacd} after some manipulations (noting that if $a,b$ correspond to $c,d$, then $-a,a+b$ correspond to $c,-c-d$).

Following the idea from \cite{MR0260674}, we wish to show that $t\mapsto P_{k,m}\hpt\in\RR$ has many sign changes, which forces the existence of many zeros on $\mathcal{L}_\rho$.
For this, we wish to show that the "main term" responsible for the oscillatory behavior of $P_{k,m}\hpt$ comes from $\sigma_{1,0}$ and $\sigma_{1,-1}$.
And so, we denote
$$
M(t) = 
\sigma_{1,0}\hpt + 
\sigma_{1,-1}\hpt.
$$
From \eqref{def-sigmacd} we see that $M(t) = 2R_M(t) \cos\left(\Theta_M(t)\right)$ with
\begin{equation}\label{def-rm}  
R_M(t) = \exp{ \frac{-2\pi mt}{\frac{1}{4} + t^2}}
\left(\frac{1}{4} + t^2\right)^{-k/2}
\end{equation}

and
\begin{equation}\label{def-thetam}
\Theta_M(t) = 
\frac{\pi m}{\frac{1}{4} + t^2} +
 k \cdot \arctan{2t}.
\end{equation}

We will see that as $t$ goes from $\mint$ to $\infty$, the function $\Theta_M(t)$ varies enough so that $M(t)$ has many sign changes.
We then wish to show that the rest of $P_{k,m}\hpt$ is significantly smaller than $R_M(t)$, which will force $P_{k,m}\hpt$ to have many sign changes as well.
That is, we wish to show that
\begin{equation}\label{eq-small_diff}  
\left| P_{k,m}\hpt - M(t)\right| = o\left(R_M(t)\right).
\end{equation}
We will see however that \eqref{eq-small_diff} does not hold for all $t\geq\mint$.
For $t$'s close to $\mint$, other terms of the form $\sigma_{c,d}\hpt$ can be larger than $R_M(t)$.
For example, when $\alpha \leq \frac{1}{2\pi\sqrt{3}}$, we will see that $\left|\sigma_{0,1}\hpt\right| =  \exp{- 2\pi m t}$ is going to be larger than $R_M(t)$ for such $t$'s.
And so, we will consider only sufficiently large $t$'s where we will show that \eqref{eq-small_diff} holds.

\subsection{Proof of \autoref{thm-lrho}}
We denote 
$$
r_M(t) = r_{1,0}\hpt = r_{1,-1}\hpt
$$
so that on any compact subset $\mathcal{C}\subset \mathcal{F}$ we have $R_M(t) = \left(r_M(t)(1 + o(1))\right)^{k}$.
We denote also
$$
\mathcal{I} = \SET{(c,d) \, : \, (c,d) \neq (1,0),(1,-1)}.
$$

We begin by giving two propositions establishing sub-segments of $\mathcal{L}_\rho$ where $r_M(t) > r_{c,d}\hpt$ for all $(c,d)\in\mathcal{I}$.
We will then use \autoref{lem-main} to show that \eqref{eq-small_diff} holds in these segments, and calculate the number of sign changes of $M(t)$ in these segments.

\begin{proposition}\label{prop-r01<rm}
Let 
\begin{equation}\label{def-F}
F(t) = 
\frac{1}{4\pi}
\frac{\qpt\log\qpt}{t\left(t^2 - \frac{3}{4}\right)}.
\end{equation}
If $\alpha \leq \frac{1}{2\pi \sqrt{3}}$, let $A,B$ be such that $F^{-1}(\alpha) < A < B$.
Otherwise, if $\alpha > \frac{1}{2\pi \sqrt{3}}$, let $A,B$ be such that $\mint < A < B$.
Then we have that $r_{0,1}\hpt < r_M(t)$ for all $t\in[A,B]$.
\end{proposition}

\begin{remark}
The function $F$ is decreasing on $\left[\frac{\sqrt{3}}{2},\infty\right)$.
Its value at $\frac{\sqrt{3}}{2}$ is $\frac{1}{2\pi\sqrt{3}}$, and its value at $\infty$ is $0$.
\end{remark}

\begin{proof}
We have that
$$
r_{0,1}\hpt = \exp{-2\pi \alpha t}
$$
and
$$
r_M(t) = \exp{-2\pi \alpha \frac{t}{\frac{1}{4} + t^2}}\qpt^{-\frac{1}{2}}.
$$
After taking log, the condition $r_{0,1}\hpt < r_M(t)$ becomes
$$
-2\pi \alpha t
<
-2\pi \alpha \frac{t}{\frac{1}{4} + t^2}
- \frac{1}{2}\log\qpt.
$$
This simplifies to $F(t) < \alpha$, or equivalently: $t > F^{-1}(\alpha)$.
For $\alpha > \frac{1}{2\sqrt{3}\pi}$ the condition above is satisfied for all $t>\mint$.
\end{proof}

\begin{proposition}\label{prop-rcd<rm}
For $\alpha \leq \frac{\sqrt{3}\log(3)}{4\pi}$ we define $T = \frac{\sqrt{3}}{2}$.
For $\alpha > \frac{\sqrt{3}\log(3)}{4\pi}$ we define $T$ to be the largest real solution in $t$ to:
$$
2\pi\alpha \frac{2t}{\qpt\left(\frac{1}{4} + t^2 + 2\right)}-\frac{1}{2}\log\left( 1 + \frac{2}{\frac{1}{4} + t^2}\right) = 0,
$$
noting that we have $T > \frac{\sqrt{3}}{2}$.

Then for $T < A< B$, we have that:
$$
r_{c,d}\hpt < r_{1,0}\hpt = r_{1,-1}\hpt
$$
for all $(c,d)\neq(1,0),(1,-1),(0,1)$ and all $t\in[A,B]$.
\end{proposition}

\begin{proof}
Using \eqref{def-rcd}, we find that the condition $r_{c,d}\hpt < r_{1,0}\hpt$ is equivalent (after taking log) to: 
\begin{equation*}
2\pi\alpha\frac{t}{\frac{1}{4} + t^2}
\frac{\qpt (c^2 - 1) + cd + d^2}{\qpt c^2 + cd + d^2}
-\frac{1}{2}\log\left( \frac{\qpt c^2 + cd + d^2}{\frac{1}{4} + t^2}\right) < 0.
\end{equation*}
Denote $s = c^2 + \frac{cd + d^2}{\frac{1}{4} + t^2}$.
Then the last expression can be written as 
\begin{equation*}
\frac{s-1}{s}\left( \frac{2\pi\alpha t}{\frac{1}{4} + t^2} - \frac{1}{2}\frac{\log(s)}{(s-1)/s}\right).
\end{equation*}

The function $\frac{\log(s)}{(s-1)/s}$ increases with $s$.
For any given $t>\mint$, amongst all $(c,d)\neq(1,0),(1,-1),(0,1)$, the value $s = c^2 + \frac{cd + d^2}{\frac{1}{4} + t^2}$ is minimal for $(c,d) = (1,1)$ or $(1,-2)$ for which we have $s = 1 + \frac{2}{\frac{1}{4} + t^2}$.
However, in this case, the expression
$$
 \frac{2\pi\alpha t}{\frac{1}{4} + t^2} - \frac{1}{2}\frac{\log(s)}{(s-1)/s}
$$
will be negative for all $t> T$.
Indeed, for $\alpha\leq \frac{\sqrt{3}\log(3)}{4\pi}$ we have
$$
 \frac{2\pi\alpha t}{\frac{1}{4} + t^2} - \frac{1}{2}\frac{\log(s)}{(s-1)/s}
 < 
 \frac{\sqrt{3} \log(3) t}{2(\frac{1}{4} + t^2)} - \frac{1}{2}\frac{\log(s)}{(s-1)/s}
$$
which is negative for all $t> \frac{\sqrt{3}}{2}$ following a straightforward computation.
For $\alpha > \frac{\sqrt{3}\log(3)}{4\pi}$ the negativity follows from the definition of $T$ and noting that the expression is negative for all sufficiently large $t$.
In either case we have that $r_{c,d}\hpt < r_{1,0}\hpt$ for all $(c,d) \neq (1,0),(1,-1), (0,1)$ and all $t > T$.
\end{proof}

We are now ready to prove \autoref{thm-lrho}.
\begin{proof}
Let $\alpha > 0$, and let $\mint< A<B$ be chosen such that the hypotheses of \autoref{prop-r01<rm} and of \autoref{prop-rcd<rm} are satisfied.
Applying \autoref{lem-main} with the compact set $\SET{\hpt \, : \, t\in[A,B]}$, the set of indices $\mathcal{I}$, and the maximal term $r_{1,0} = r_{1,-1}$, we get
$$
\left| 
P_{k,m}\hpt - M(t) 
\right| 
= o\left(R_M(t)\right)
$$
uniformly for $t\in [A,B]$.
And so, we want to understand how many sign changes 
$$
M(t) = 2R_M(t) \cos\left(\Theta_M(t)\right)
$$ 
has in this range.

We denote
$$
\vartheta(t) = \frac{\pi \alpha}{\frac{1}{4} + t^2} + \arctan{2t}
$$
so that $\Theta_M(t) = k\vartheta(t) + o(k)$.
We also denote $V_A^B(\vartheta)$ the total variation of $\vartheta$ from $A$ to $B$.
With these notations, the number of sign changes of $M(t)$ in $[A,B]$ is $\frac{k}{\pi} V_A^B(\vartheta) + o(k)$, which also implies (as a lower bound) the same amount of sign changes in $P_{k,m}\hpt$.

We first consider the case where $\alpha \leq \frac{1}{2\sqrt{3}\pi}$.
In this case, the function $\vartheta$ is increasing on $t\geq\mint$, so that $V_A^B(\vartheta) = \vartheta(B) - \vartheta(A)$.
Thus, the number of zeros of $P_{k,m}$ on $\mathcal{L}_\rho$ is at least 
$$
\frac{k}{\pi} V_A^B(\vartheta)  + o(k).
$$
For $\alpha \leq \frac{1}{2\sqrt{3}\pi}$ we can choose any $A<B$ satisfying $F^{-1}(\alpha) < A< B$.
It follows then that the number of zeros of $P_{k,m}$ on $\mathcal{L}_\rho$ is at least 
$$
\left(\frac{12}{\pi} V_{F^{-1}(\alpha)}^\infty(\vartheta)\right) \frac{k}{12} + o(k).
$$
We now calculate,
$$
V_{F^{-1}(\alpha)}^\infty
=
\vartheta(\infty) - \vartheta(F^{-1}(\alpha))
=
\frac{\pi}{2} - \frac{\pi \alpha}{\frac{1}{4} + F^{-1}(\alpha)^2} - \arctan{2F^{-1}(\alpha)}
$$
and so, in the case $\alpha \leq \frac{1}{2\sqrt{3}\pi}$, we can set 
$$
P_\rho(\alpha) = 
12\left(\frac{1}{2} - \frac{ \alpha}{\frac{1}{4} + F^{-1}(\alpha)^2} - \frac{1}{\pi}\arctan{2F^{-1}(\alpha)}\right).
$$

We now consider the case $\frac{1}{2\pi \sqrt{3}} < \alpha \leq \frac{\sqrt{3}\log(3)}{4\pi}$.
As before, we have that the number of zeroes of $P_{k,m}$ on $\mathcal{L}_\rho$ is at least $\frac{k}{\pi} V_A^B(\vartheta) + o(k)$.
This time we can choose any $A<B$ with $A > \mint$.
And so we get that the number of zeroes on $\mathcal{L}_\rho$ is at least 
$$
\frac{k}{\pi}V_{\mint}^{\infty}(\vartheta) + o(k).
$$
This time, the function $\vartheta$ is first decreasing and then increasing on $t\geq\mint$, with the minimum attained at 
$$
t_\text{min} = 2\pi \alpha + \sqrt{4\pi^2 \alpha^2 - \frac{1}{4}}
$$
at which the value of $\vartheta$ is 
$$
\vartheta(t_\text{min})
=
\frac{\pi\alpha}{\frac{1}{4} + t_\text{min}^2} + \arctan{2t_\text{min}}.
$$
It follows that 
\begin{multline*}
V_{\mint}^\infty = 
\left(\vartheta\left(\mint\right) - \vartheta(t_\text{min})\right) 
+ 
\left(\vartheta(\infty) - \vartheta(t_\text{min})\right)
=\\
\left(\pi\alpha + \frac{\pi}{3} -
\vartheta(t_\text{min})\right)
+
\left(\frac{\pi }{2} - \vartheta(t_\text{min})\right).
\end{multline*}
And so, in this case, we can set
\begin{multline*}
P_\rho(\alpha) = 
\frac{12}{\pi}\left(\pi\alpha + \frac{\pi}{3} - 
 \vartheta(t_\text{min})\right) +
\frac{12}{\pi}\left(\frac{\pi }{2} - \vartheta(t_\text{min})\right) = \\
12\left(\frac{5}{6} + \alpha - \frac{2}{\pi}\vartheta(t_\text{min})\right).
\end{multline*}
A plot of the values $P_\rho(\alpha)$ described above is given in \autoref{fig-palpha}. 
\begin{figure}[ht]
\centering
\includegraphics[trim=2cm 0 2cm 2cm, width=1\textwidth]{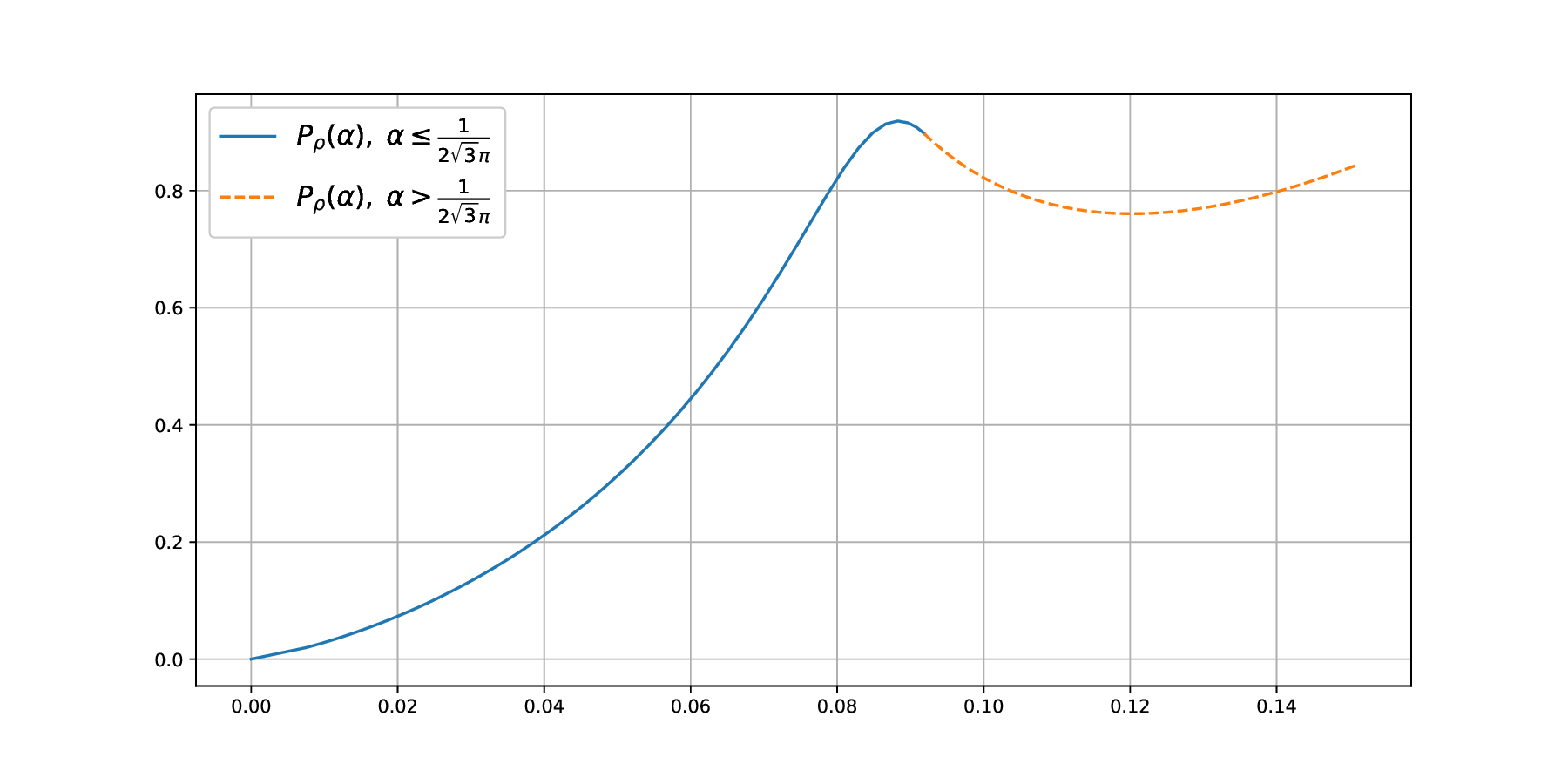}
\caption{The values of $P_\rho(\alpha)$ described in the proof, for the ranges $0<\alpha \leq \frac{1}{2\sqrt{3}\pi}$ and $\frac{1}{2\sqrt{3}\pi} < \alpha \leq \frac{\sqrt{3}\log(3)}{4\pi}$.}\label{fig-palpha}
\end{figure}

We now consider the case where $\alpha > \frac{\sqrt{3}\log(3)}{4\pi}$.
Let $T > \mint$ be as in \autoref{prop-rcd<rm}.
As in the previous cases, we have that the number of zeroes of $P_{k,m}$ on $\mathcal{L}_\rho$ is at least $\frac{k}{\pi} V_A^B(\vartheta) + o(k)$ for any $T<A<B$, which implies that the number of such zeros is at least $\frac{k}{\pi} V_T^\infty(\vartheta) + o(k)$.
Since $V_{T}^\infty(\vartheta) > 0$, we conclude that the number of zeros of $P_{k,m}$ on $\mathcal{L}_\rho$, which we denote $N_{\mathcal{L}_\rho}$, satisfies $N_{\mathcal{L}_\rho}\gg k$.
This proves that a positive proportion of the zeros of $P_{k,m}$ lie on $\mathcal{L}_\rho$.

\end{proof}

 % zeros on L_\rho

\section{Zeros on \texorpdfstring{$\mathcal{L}_i$}{the line it}}\label{sec-4}

\subsection{Sketch of proof}
As before, we have that 
\begin{equation*}
P_{k,m}(it) = \sigma_{0,1}(it) + \sum_{\substack{c,d\in\ZZ \\ c>0 \\ \gcd(c,d) = 1}}\sigma_{c,d}(it).
\end{equation*}
A simple calculation shows that
$$
\sigma_{c,d}\left(it\right) = 
\exp{\frac{-2\pi m t}{c^2 t^2 + d^2}} 
\exp{2\pi i m\frac{ac t^2 + bd}{c^2 t^2 + d^2}}
\left(cti + d \right)^{-k}.
$$
Once more it can be seen directly from this definition that $P_{k,m}(it)$ is real, since
$$
\sigma_{c,d}(it) = \overline{\sigma_{c,-d}(it)}.
$$

A similar argument to the case of $\mathcal{L}_\rho$ shows that for $t$ large enough, the sum in \eqref{def-pkm} is dominated by $\left|\sigma_{1,0}(it)\right|$.
However, unlike the case of $\mathcal{L}_\rho$, this time 
$$
\sigma_{1,0}(it) = (ti)^{-k}\exp{-2\pi m \frac{1}{t}}
$$
does not oscillate, and so does not contribute any sign changes.
Thus, in order to find a large amount of sign changes on $\mathcal{L}_i$, we must consider smaller $t$'s where other terms in \eqref{def-pkm} are dominant.

We are going to show that as $t$ goes from $1$ to $\infty$, the indices $(c,d)$ for which $r_{c,d}(it)$ is maximal can change several times before stabilizing on $(1,0)$.
We show that if $\alpha > \frac{\log(2)}{2\pi}$,  then there is some interval of time $t\in [A,B]$,  $1< A  < B$, such that in this interval $r_{1,\pm 1}(it)$ are maximal amongst all other $r_{c,d}$'s.
Since $\sigma_{1, 1}(it) + \sigma_{1, -1}(it)$ exhibits oscillatory behavior, we will conclude that $P_{k,m}(it)$ has many sign changes in this segment, which will show that a positive proportion of the zeros lie on $\mathcal{L}_i$.

\subsection{Proof of \autoref{thm-li}}
\begin{proof}
We denote
$$
M(t) = \sigma_{1,1}(it) + \sigma_{1,-1}(it).
$$
We see that $M(t) = 2R_M(t)\cos\left(\Theta_M(t)\right)$ with
$$
R_M(t) = \frac{\exp{-2\pi m \frac{t}{t^2+1}}}{\left(t^2 + 1\right)^{k/2}}
$$
and
$$
\Theta_M(t) = 2\pi m\frac{t^2 }{t^2 + 1} - k\arctan{t}
=
k\left( \frac{2\pi \alpha t^2  }{t^2 + 1} - \arctan{t}\right) + o(k)
.
$$

From \eqref{def-Iw}, we see that for a given $z$, $r_{c,d}(z)$ is maximal amongst all $c,d$'s if $I(w_{c,d})$ is maximal.
The function $I(w) = -2\pi\alpha w  + \frac{1}{2}\log(w)$ is maximal at $w = (4\pi\alpha)^{-1}$.
Assume first that $\alpha > \frac{1}{2\pi}$.
Let $t_0$ be the larger solution to 
$$
w_{1,\pm 1}(it)=
\frac{t}{t^2 + 1} = 
(4\pi\alpha)^{-1}
$$ 
Note that $t_0> 1$ when $\alpha > \frac{1}{2\pi}$.
Thus, at $t_0$, $r_{1,\pm 1}(it)$ are maximal amongst all other $r_{c,d}(it)$ (since $w_{c,d}(it_0) \neq (4\pi\alpha)^{-1}$ for all other $c,d$).
From continuity, we can find a segment $[A,B]$ around $t_0$ such that 
\begin{equation}\label{eq-r11_max}
r_{1,\pm 1}(it) > r_{c,d}(it)
\end{equation}
for all other $c,d$'s and for all $t\in[A,B]$.

For $\frac{\log(2)}{2\pi }<\alpha \leq \frac{1}{2\pi}$, it can be easily verified that at $t = 1$: $I(w_{c,d}(i\cdot 1))$ is maximal for $(c,d) = (1,\pm1)$.
Thus, from continuity, we can find $1 < A < B$ such that \eqref{eq-r11_max} is satisfied for all $(c,d)\neq(1,\pm1)$ and for all $t\in[A,B]$.
In either case we found a segment $[A,B]\subset (1,\infty)$ where \eqref{eq-r11_max} holds.

We now apply \autoref{lem-main} with the compact set $\SET{it\,:\, t\in [A,B]}$, and the set of indices $\mathcal{I} = \SET{(c,d) \, : \, (c,d) \neq (1,\pm 1)}$.
This gives  
$$
\left|P_{k,m}(it) - M(t)\right| = o(R_M(t))
$$
uniformly in $t\in [A,B]$.
And so, the number of sign changes of $P_{k,m}(it)$ in $[A,B]$ is bounded below by the number of sign changes of $M(t)$ in this segment.
Denote
$$
\vartheta(t) = \frac{2\pi \alpha t^2  }{t^2 + 1} - \arctan{t}
$$
so that $\Theta_M(t) = k\vartheta(t) + o(k)$.
Denote by $V_A^B(\vartheta)$ the total variation of $\vartheta$ in $[A,B]$.
Then the number of sign changes of $\cos\left(\Theta_M(t)\right)$ in $[A,B]$ is given by
$$
\frac{k}{\pi} V_A^B(\vartheta) + o(k).
$$
It follows that the number of zeros of $P_{k,m}$ on $\mathcal{L}_i$, which we denote by $N_{\mathcal{L}_i}$, is at least $\frac{k}{\pi} V_A^B(\vartheta) + o(k)$.
Since $V_A^B(\vartheta) > 0$, we conclude that $N_{\mathcal{L}_i}\gg k$.
And so, a positive proportion of the zeros of $P_{k,m}$ in $\mathcal{F}$ are on $\mathcal{L}_i$.
\end{proof}
 % zeros on L_i

\section{Non-real zeros}\label{sec-5}
Other than the real zeros in \autoref{fig-zeros}, the non-real zeros also seem to be following a pattern.
In this section explain this phenomenon when $\alpha < \frac{1}{4\pi}$.

We remark that the method of proof can work for larger $\alpha$'s as well, and can also be applied to deduce the asymptotic density of zeros along the curve $\Gamma_\alpha$. 
We choose to restrict to the case $\alpha < \frac{1}{4\pi}$ as this simplifies many of the computations.

\subsection{Sketch of proof}
In the limit $m\sim\alpha k$, a consequence of \autoref{lem-main} is that $P_{k,m}(z)$ becomes well approximated by those $\sigma_{c,d}$'s for which $r_{c,d}$ is maximal.
In particular, for a compact set where only a single $r_{c,d}$ is maximal, we have the following:
\begin{lemma}\label{lem-nonzero_compact}
Let $\mathcal{C}\subset \mathcal{F}$ is a compact set. 
Assume that $r_{c_m,d_m}$ is maximal amongst all other $r_{c,d}$ on $\mathcal{C}$.
Then $P_{k,m}(z) = \sigma_{c_m,d_m}(z)(1 + o(1))$ uniformly in $\mathcal{C}$, and as a consequence $P_{k,m}(z) \neq 0$ in $\mathcal{C}$ for large enough $k$.
\end{lemma}
\begin{proof}
The fact that $P_{k,m}(z) = \sigma_{c_m,d_m}(z)(1 + o(1))$ follows from \autoref{lem-main}. 
the consequence $P_{k,m}(z) \neq 0$ then follows from $\sigma_{c_m,d_m}(z) \neq 0$.
\end{proof}

Thus, it is reasonable to guess that the zeros of $P_{k,m}$ in $\mathcal{F}$ will cluster around segments where $r_{c_1,d_1}(z) = r_{c_2,d_2}(z)$ for two different pairs $(c_1,d_1)$ and $(c_2,d_2)$, where both are maximal. 
Note that this does not immediately follow from \autoref{lem-nonzero_compact}, since the zeroes can still be 'very high up' near the cusp. Nor does it prove that every segment where two different terms are maximal at once must have zeros nearby.

We will however show that for $\alpha < \frac{1}{4\pi}$, the zeros do indeed cluster around the segments where two $r_{c,d}$'s are maximal at once. 
For $\alpha < \frac{1}{4\pi}$ there are four relevant $(c,d)$ pairs to consider: $\sigma_{0,1}$, $\sigma_{1,0}$, $\sigma_{1,1}$ and $\sigma_{1,-1}$.
The area where each is dominant is shown in \autoref{fig-max}.
\begin{figure}[ht]
\centering
\includegraphics[trim= 0 0 0 0 , width=1\textwidth]{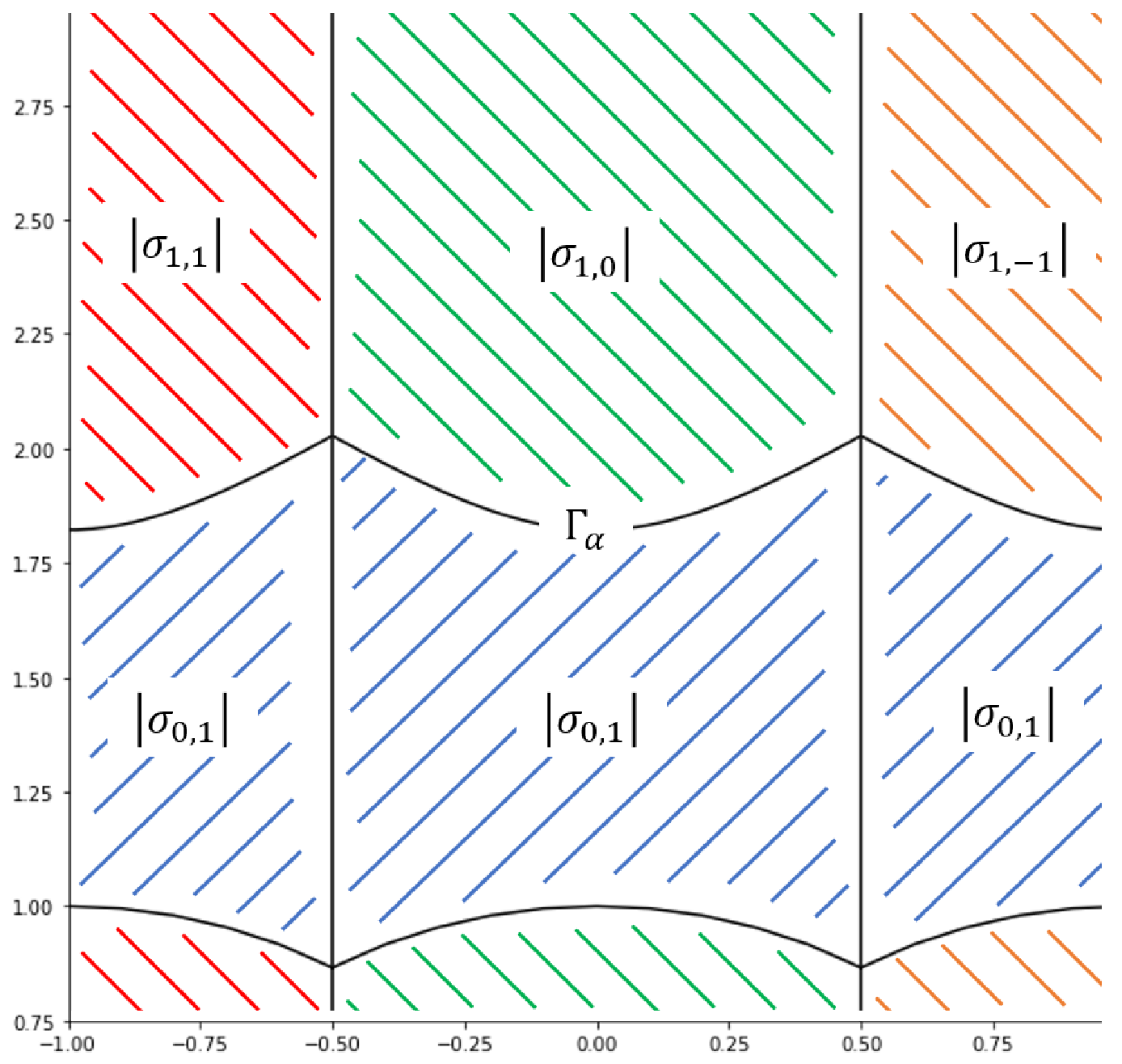}
\caption {Showing the dominant $|\sigma_{c,d}|$ for $P_{1200,90}$ in each area in $\mathcal{F}$ and its surrounding. 
}\label{fig-max}
\end{figure}
This matches with \autoref{fig-zeros} where the zeros of $P_{1200,90}$ seem to be on the lines where two different $|\sigma_{c,d}|$'s are maximal at once.
The curve $\Gamma_\alpha$ defined in the statement of \autoref{thm-non_real} is precisely the set of those $z\in\mathcal{F}\setminus\mathcal{A}$ for which $r_{1,0}(z) = r_{0,1}(z)$.
The condition $\alpha < \frac{1}{4\pi}$ ensures that $\Gamma_\alpha$ is connected and does not intersect $\mathcal{A}$.

Our method up until this point was to look at $P_{k,m}(z)$ along the relevant curve, and use the fact that this function was real and changed signs many times to deduce the existence of many zeros.
In this situation however, the function $P_{k,m}(z)$ restricted to $\Gamma_\alpha$ does not seem to be real in any obvious way.
In fact, we believe the zeros of $P_{k,m}$ are not on $\Gamma_\alpha$ (as was the case with $\mathcal{A}$, $\mathcal{L}_\rho$, $\mathcal{L}_i$), but only tend towards $\Gamma_\alpha$ as $k$ grows.
Thus, we must consider a different approach than in the previous cases.
We instead use Cauchy's argument principle, which will be able to capture zeroes in a neighborhood of $\Gamma_\alpha$.
That is, we will consider
\begin{equation}\label{eq-oint}
\frac{1}{2\pi i}\oint_\Upsilon \frac{P_{k,m}'(z)}{P_{k,m}(z)}    dz
\end{equation}
for an appropriate contour $\Upsilon$ covering most of $\Gamma_\alpha$.
The guiding idea is that $P_{k,m}$ behaves like $\sigma_{1,0}$ above $\Gamma_\alpha$, and behaves like $\sigma_{0,1}$ below $\Gamma_\alpha$.
Thus, we can hope that \eqref{eq-oint} will be large due to this discrepancy. 

\subsection{Proof of \autoref{thm-non_real}}

In order to apply Cauchy's argument principle, we will also need an analogue of \autoref{lem-main} for $P_{k,m}'$.

\begin{lemma}\label{lem-derivative}
Let $\mathcal{C}\subset \mathcal{F}$ be a compact set. 
Let $\mathcal{I}$ be a subset of the indices $(c,d)$.
Assume that for some $(c_m,d_m)\not \in \mathcal{I}$ we have that $r_{c_m,d_m}(z) > r_{c,d}(z)$ for all $(c,d)\in\mathcal{I}$ and all $z\in \mathcal{C}$.
Then 
$$
\sum_{(c,d)\in\mathcal{I}}\left|\sigma'_{c,d}(z)\right|
= o\left(\left|\sigma'_{c_m,d_m}(z)\right|\right)
$$
with implied constants depending on $c_m,d_m$.
\end{lemma}

\begin{proof}
From \eqref{def-sigmacd} we have that
$$
\sigma_{c,d}'(z) = 
- \frac{\sigma_{c,d}(z)}{cz + d}
\left(
k - \frac{1}{cz + d}
\right).
$$
It follows then that 
$$
\left|\sigma_{c,d}'(z)\right|
\ll 
k \left|\sigma_{c,d}(z)\right|
$$
and that
\begin{equation*}
\left|\sigma_{c_m,d_m}'(z)\right|
\gg
k  \left|\sigma_{c_m,d_m}(z)\right|
\end{equation*}
with implied constants depending on $c_m,d_m$.
From \autoref{lem-main} we get
$$
\sum_{(c,d)\in\mathcal{I}}
\left|\sigma_{c,d}(z)\right| = o(\left|\sigma_{c_m,d_m}(z)\right|).
$$
Thus, we have
\begin{equation*}
\sum_{(c,d)\in\mathcal{I}}\left|\sigma_{c,d}'(z)\right|
\ll \sum_{(c,d)\in\mathcal{I}}
\left|k\sigma_{c,d}(z)\right|
= o\left(k\left|\sigma_{c_m,d_m}(z)\right|\right)
= o\left(\left|\sigma_{c_m,d_m}'(z)\right|\right) 
\end{equation*}
with implied constants depending on $c_m,d_m$.
\end{proof}

\begin{corollary}\label{cor-derivative}
Let $\mathcal{C}\subset \mathcal{F}$ be a compact set.
Assume that $r_{c_m,d_m}$ is maximal amongst all other $r_{c,d}$ on $\mathcal{C}$.
Then $P'_{k,m}(z) = \sigma'_{c_m,d_m}(z)(1 + o(1))$ uniformly in $\mathcal{C}$.
\end{corollary}
\begin{proof}
We have that
$$
P_{k,m}'(z) = \sum_{(c,d)} \sigma_{c,d}'(z).
$$
The result then follows from \autoref{lem-derivative} with $\mathcal{I} = \SET{(c,d)\neq(c_m,d_m)}$.
\end{proof}

\begin{proof}[Proof of \autoref{thm-non_real}]
Fix some $\varepsilon > 0$ sufficiently small. 
Throughout the proof, the notation $\BigO{\varepsilon}$ will be meant as $\varepsilon \cdot \BigO{1}$ with $\BigO{1}$ not depending on $\varepsilon$.
All other implied constants throughout the proof can depend on $\varepsilon$.

We say that a point $x + yi$ is $\varepsilon$ below $\Gamma_\alpha$ if $x + (y+\delta)i\in\Gamma_\alpha$ for some $\delta \geq \varepsilon$.
We similarly define a point to be $\varepsilon$ above $\Gamma_\alpha$.

We denote 
\begin{equation*}
\widetilde{r_{c,d}}(z) = \frac{\exp{-2\pi\frac{m}{k} \frac{y}{\left| cz + d\right|^2}}}{\left|cz + d\right|},
\end{equation*} 
\begin{equation*}
\widetilde{\theta_{c,d}}(z)
= 
2\pi \frac{m}{k} \frac{ac |z| + (ad + bc)x + bd}{\left| cz + d\right|^2}
+ \arg{cz+d},
\end{equation*}
and 
\begin{equation*}
\theta_{c,d}(z)
= 
2\pi \alpha \frac{ac |z| + (ad + bc)x + bd}{\left| cz + d\right|^2}
+ \arg{cz+d}.
\end{equation*}
With these notations we have $\sigma_{c,d}(z) = \left(\widetilde{r_{c,d}}(z) e^{i \widetilde{\theta_{c,d}}(z)}\right)^k$.

Unlike $r_{c,d}(z)$, $\theta_{c,d}(z)$, the functions $\widetilde{r_{c,d}}(z), \widetilde{\theta_{c,d}}(z)$, depend on $k$.
However, we note that we have
$$
\widetilde{r_{c,d}}(z) = r_{c,d}(z)(1 + o(1)),\quad
\widetilde{\theta_{c,d}}(z) = \theta_{c,d}(z)(1 + o(1))
$$
on any compact subset $\mathcal{C}\subset \mathcal{F}$.

Specifically, we will need
$$
\widetilde{\theta_{0,1}}(z) = 2\pi  \frac{m}{k} x,\quad 
\widetilde{\theta_{1,0}}(z) = -2\pi  \frac{m}{k} \frac{x}{|z|^2} + \arg{z},
$$
and
$$
\widetilde{r_{0,1}}(z) =\exp{-2\pi \frac{m}{k}y}
,\quad 
\widetilde{r_{1,0}}(z) = \frac{\exp{-2\pi \frac{m}{k}\frac{y}{|z|^2}}}{|z|}.
$$

Consider the contour $\Upsilon$ depicted in \autoref{fig-contour} which we now describe.
\begin{figure}[ht]
\centering
\includegraphics[trim=0 0 0 0, width=1\textwidth]{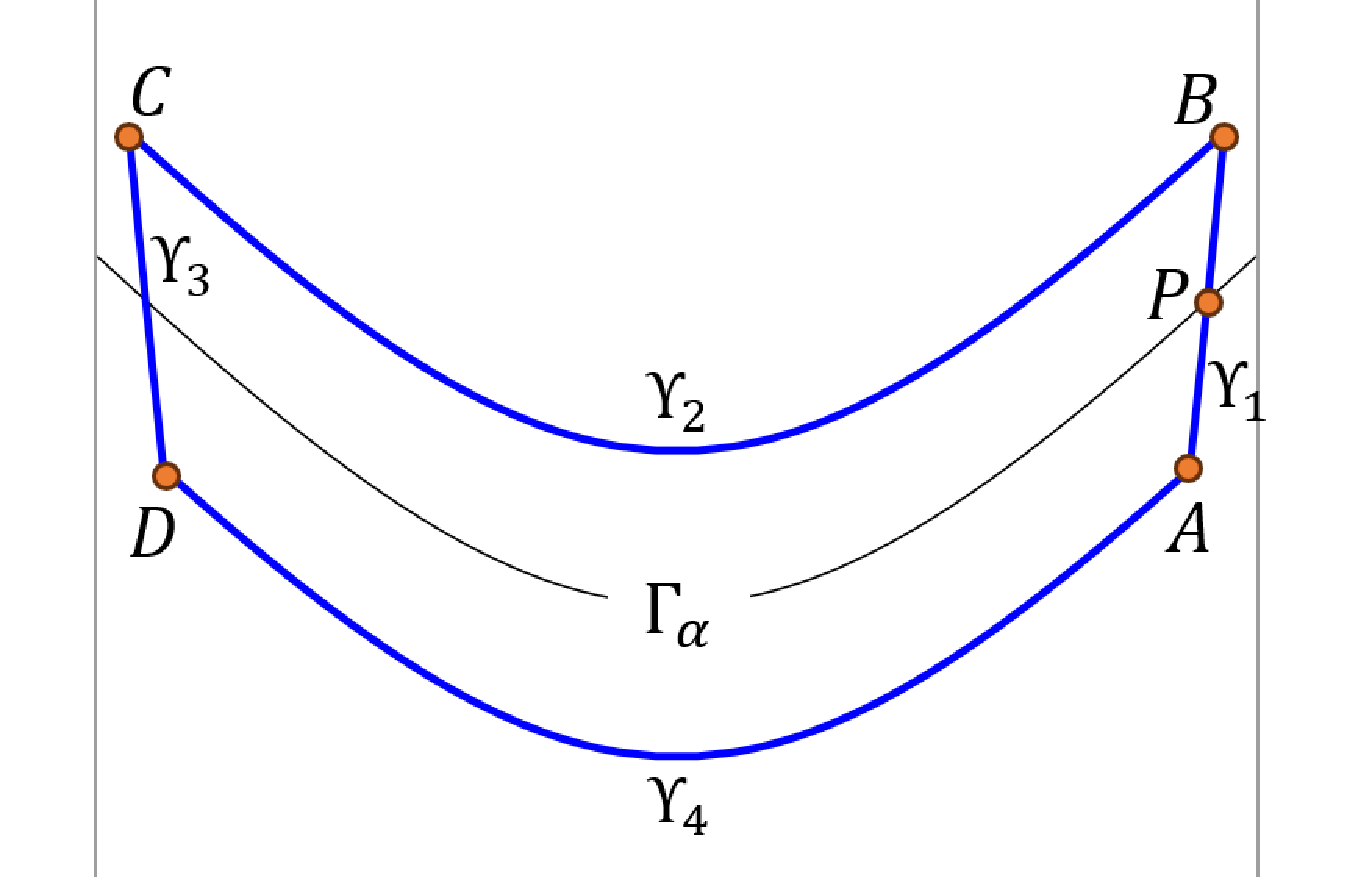}
\caption {The contour $\Upsilon$. 
}\label{fig-contour}
\end{figure}
The curve $\Gamma_\alpha$ intersects $\mathcal{L}_\rho$ at $\frac{1}{2} + F^{-1}(\alpha) i$ where $F$ was defined in \autoref{prop-r01<rm}.
We pick a point $P$ on $\Gamma_\alpha$ whose distance $d$ from $\mathcal{L}_\rho$ satisfies $\varepsilon \ll d \ll \varepsilon$, so that we can write
$$
P = \left(\frac{1}{2} + \BigO{\varepsilon}\right)
 + \left(F^{-1}(\alpha) + \BigO{\varepsilon}\right)i.
$$
We further adjust $P$ such that at $P$ we will have 
\begin{equation}\label{eq-theta_diff_cond}
\widetilde{\theta_{0,1}}(P) - \widetilde{\theta_{1,0}}(P) = v
\in \frac{\pi}{k}\ZZ.
\end{equation}
Note that we have 
$$
\widetilde{\theta_{0,1}}(z) - \widetilde{\theta_{1,0}}(z)
=
\left(\theta_{0,1}(z) - \theta_{1,0}(z)\right)(1 + o(1))
$$
so that the function $\widetilde{\theta_{0,1}}(z) - \widetilde{\theta_{1,0}}(z)$ does not vary much with $k$.
Furthermore, the function $\widetilde{\theta_{0,1}}(z) - \widetilde{\theta_{1,0}}(z)$ is not constant on any sub-segment of $\Gamma_\alpha$ with positive length (for this would imply that $\sigma_{0,1}(z) = \sigma_{1,0}(z)e^{iC}$ on this segment for some constant $C\in\RR$, which would imply that these functions are equal on all of $\HH$, which is false).
Thus, for $k$ large, the collection of points $P$ on $\Gamma_\alpha$ satisfying \eqref{eq-theta_diff_cond} becomes dense, which means that we can modify our original choice of $P$ by $o(1)$ so that \eqref{eq-theta_diff_cond} will be satisfied as well.

The curve $\Upsilon_1$ is given by the condition 
$$
\widetilde{\theta_{0,1}}(z) - \widetilde{\theta_{1,0}}(z) = v,
$$
which passes through $P$.
The endpoints $A$, $B$ are chosen in the following way.
We require that the point $A$ is $\varepsilon$ below $\Gamma_\alpha$, and we also require that the distance of $A$ from the boundary of $\mathcal{F}$ is at least $\varepsilon$.
Similarly, we require that $B$ is $\varepsilon$ above $\Gamma_\alpha$ and distance at least $\varepsilon$ from the boundary of $\mathcal{F}$.
We also make sure that $\text{Length}\left(\Upsilon_1\right) = \BigO{\varepsilon}$.

The curve $\Upsilon_3$ is then the reflection of $\Upsilon_1$ about the imaginary axis.

The curve $\Upsilon_2$ connects $B$ to $C$, and we make sure that $\Upsilon_2$ is always $\varepsilon$ above $\Gamma_\alpha$ and distance at least $\varepsilon$ from $\partial \mathcal{F}$.

Similarly, the curve $\Upsilon_4$ connects $D$ to $A$ and we make sure that $\Upsilon_4$ is always $\varepsilon$ below $\Gamma_\alpha$ and distance at least $\varepsilon$ from $\partial \mathcal{F}$.
This is possible for $\alpha < \frac{1}{4\pi}$ and $\varepsilon$ sufficiently small, since in this case $\Gamma_\alpha$ is entirely above $\mathcal{A}$.

We denote $\text{Int}(\Upsilon)$ the interior of $\Upsilon$, and by $\overline{\text{Int}(\Upsilon)}$ its closure.
Denote also $\mathcal{C}_+$ the set of points in $\overline{\text{Int}(\Upsilon)}$ which are $\varepsilon$ above $\Gamma_\alpha$, and $\mathcal{C}_-$ the set of points in $\overline{\text{Int}(\Upsilon)}$ which are $\varepsilon$ below $\Gamma_\alpha$.
From \autoref{lem-nonzero_compact} and \autoref{cor-derivative} we see that on $\mathcal{C}_-$ we have
\begin{equation}\label{eq-on_c-}
P_{k,m}(z) = \sigma_{0,1}(z) (1 + o(1)),\quad
P_{k,m}'(z) = \sigma_{0,1}'(z) (1 + o(1)),
\end{equation}
and on $\mathcal{C}_+$ we have
\begin{equation}\label{eq-on_c+}
P_{k,m}(z) = \sigma_{1,0}(z) (1 + o(1)),\quad
P_{k,m}'(z) = \sigma_{1,0}'(z) (1 + o(1)).
\end{equation}

We now compute \eqref{eq-oint} by calculating the integral over the various segments in \autoref{fig-contour}.
We begin by considering $\Upsilon_1$.
On $\Upsilon_1$, either $r_{0,1}$ or $r_{1,0}$ (or both) are larger than all other $r_{c,d}$'s.
Thus, from \autoref{lem-main} we have that on $\Upsilon_1$
$$
\left|P_{k,m}(z) - \sigma_{0,1}(z) -\sigma_{1,0}(z)\right| = o\left(\left| \sigma_{0,1}(z) \right| + \left| \sigma_{1,0}(z) \right|\right).
$$
From the definition of $\Upsilon_1$, we have that on this curve $\sigma_{0,1}(z)$ and $\sigma_{1,0}(z)$ have the same argument, so that
$$
\left| \sigma_{0,1}(z) \right| + \left| \sigma_{1,0}(z) \right| = 
\left| \sigma_{0,1}(z) + \sigma_{1,0}(z) \right|.
$$
It follows that on $\Upsilon_1$ we have
$$
P_{k,m}(z) = \left(\sigma_{0,1}(z) +\sigma_{1,0}(z)\right)(1+o(1)).
$$
From \autoref{lem-derivative} we also get that on  $\Upsilon_1$:
$$
\left|P_{k,m}'(z) - \sigma_{0,1}'(z) -\sigma_{1,0}'(z)\right| = o\left(\left| \sigma_{0,1}'(z) \right| + \left| \sigma_{1,0}'(z) \right|\right)
$$
so that
$$
\left|P_{k,m}'(z)\right| \ll \left| \sigma'_{0,1}(z) \right| + \left| \sigma'_{1,0}(z) \right|.
$$
And so we have that on $\Upsilon_1$:
\begin{equation*}
\left|\frac{P'_{k,m}(z)}{P_{k,m}(z)}\right|
\ll
\frac{\left| \sigma'_{0,1}(z) \right| + \left| \sigma'_{1,0}(z) \right|}{\left| \sigma_{0,1}(z) \right| + \left| \sigma_{1,0}(z) \right|}
\ll 
\left|\frac{\sigma'_{0,1}(z)}{\sigma_{0,1}(z)}\right|
+
\left|\frac{\sigma'_{1,0}(z)}{\sigma_{1,0}(z)}\right|
\ll k.
\end{equation*}
Since $\text{Length}\left(\Upsilon_1\right) = \BigO{\varepsilon}$, we conclude that
\begin{equation}\label{eq-upsilon1}
\int_{\Upsilon_1}
\frac{P_{k,m}'(z)}{P_{k,m}(z)}dz
= \BigO{\varepsilon} k.
\end{equation}
From symmetry we also conclude
\begin{equation}\label{eq-upsilon3}
\int_{\Upsilon_3}
\frac{P_{k,m}'(z)}{P_{k,m}(z)}dz
= \BigO{\varepsilon} k.
\end{equation}

We now consider the curve $\Upsilon_4$.
From \eqref{eq-on_c-} we have that on this curve
\begin{multline*}
\frac{P_{k,m}'(z)}{P_{k,m}(z)}
= 
(1 + o(1))\frac{\sigma'_{0,1}(z)}{\sigma_{0,1}(z)}
= \\
(1+o(1)) k \frac{d}{dz}\left[\log\left(\widetilde{r_{0,1}}(z)e^{i\widetilde{\theta_{0,1}}(z)}\right)\right] 
=\\
(k + o(k))\frac{d}{dz}\left[\log\left(\exp{2\pi i \frac{m}{k}z}\right)\right] 
= \\
(k + o(k))\frac{d}{dz}\left[ 2\pi i \alpha z\right].
\end{multline*}
Noting that
\begin{gather*}
D = \left(-\frac{1}{2} + \BigO{\varepsilon}\right) +  \left(F^{-1}(\alpha) + \BigO{\varepsilon}\right)i \\
A = \left(\frac{1}{2} + \BigO{\varepsilon}\right) +  \left(F^{-1}(\alpha) + \BigO{\varepsilon}\right)i
\end{gather*}
it follows that
\begin{equation}\label{eq-upsilon4}
\int_{\Upsilon_4} \frac{P_{k,m}'(z)}{P_{k,m}(z)}
= 
(k + o(k)) 2 \pi i \alpha z\big|_{D}^{A} 
= 
k \left(2 \pi i \alpha + \BigO{\varepsilon} + o(1)\right).
\end{equation}

We now consider the curve $\Upsilon_2$.
From \eqref{eq-on_c+} we have that on this curve
\begin{multline*}
\frac{P_{k,m}'(z)}{P_{k,m}(z)}
= 
(1 + o(1))\frac{\sigma'_{1,0}(z)}{\sigma_{1,0}(z)}
= \\
(1+o(1))k \frac{d}{dz}\left[\log\left(\widetilde{r_{1,0}}(z)e^{i\widetilde{\theta_{1,0}}(z)}\right)\right] 
=\\
(k + o(k))\frac{d}{dz}\left[\log\left(\frac{\exp{2\pi i \frac{m}{k}\cdot\frac{-1}{z}}}{z}\right)\right] 
= \\
(k + o(k))\frac{d}{dz}\left[ 2\pi i \alpha \frac{-1}{z} - \log(|z|) - i \arg{z}\right].
\end{multline*}
Noting that
\begin{gather*}
B = \left(\frac{1}{2} + \BigO{\varepsilon}\right) +  \left(F^{-1}(\alpha) + \BigO{\varepsilon}\right)i \\
C = \left(-\frac{1}{2} + \BigO{\varepsilon}\right) +  \left(F^{-1}(\alpha) + \BigO{\varepsilon}\right)i
\end{gather*}
it follows that
\begin{multline}\label{eq-upsilon2}
\int_{\Upsilon_2} \frac{P_{k,m}'(z)}{P_{k,m}(z)}
= 
(k + o(k)) \left.\left[ 2\pi i \alpha \frac{-1}{z} - \log(|z|) - i \arg{z}\right] \right|_{B}^{C} 
= \\
(k + o(k))
\left(
2\pi i\alpha
\frac{1}{\frac{1}{4} + F^{-1}(\alpha)^2}
+ 2i\arctan{2F^{-1}(\alpha)}
-\pi i + \BigO{\varepsilon}
\right).
\end{multline}

Combining \eqref{eq-upsilon1}, \eqref{eq-upsilon2}, \eqref{eq-upsilon3}, \eqref{eq-upsilon4},
we get
\begin{multline*}
\frac{1}{2\pi i}\oint_\Upsilon \frac{P_{k,m}'(z)}{P_{k,m}(z)}    dz
=\\
(k+o(k))\left(
\alpha
+ \frac{\alpha}{\frac{1}{4} + F^{-1}(\alpha)^2}
+ \frac{1}{\pi} \arctan{2 F^{-1}(\alpha)} - \frac{1}{2}
+ \BigO{\varepsilon}
\right).
\end{multline*}
By the argument principle, this is the number of zeros in the interior of $\Upsilon$.
However, from \autoref{lem-nonzero_compact} we can conclude that all of these zeros are $\varepsilon$ near to $\Gamma_\alpha$, since $P_{k,m}$ has no zeros in $\mathcal{C}_+$, $\mathcal{C}_-$ for large enough $k$.

It remains to show that this accounts for almost all zeros not on $\mathcal{L}_\rho$ or $\mathcal{A}$.
Denote $N_{\mathcal{L}_\rho}$, $N_{\mathcal{A}}$ the number of zeros of $P_{k,m}$ on $\mathcal{L}_\rho$ and on $\mathcal{A}$ respectively.
Denote also $N_{\Gamma_\alpha, \varepsilon}$ the number of non-real zeros of $P_{k,m}$ which are $\varepsilon$ near $\Gamma_\alpha$.
In our range of $\alpha < \frac{1}{4\pi}$, we have from \autoref{thm-lrho} that
$$
N_{\mathcal{L}_\rho} \gtrsim
k\left(\frac{1}{2} - \frac{ \alpha}{\frac{1}{4} + F^{-1}(\alpha)^2} - \frac{1}{\pi}\arctan{2F^{-1}(\alpha)}\right).
$$
Rankin's result in \cite{MR0664646} gives us 
$$
N_{\mathcal{A}} \gtrsim k\left(\frac{1}{12} - \alpha\right),
$$
and we have just shown
$$
N_{\Gamma_\alpha, \varepsilon}
\gtrsim
k\left(
\alpha
+ \frac{\alpha}{\frac{1}{4} + F^{-1}(\alpha)^2}
+ \frac{1}{\pi} \arctan{2 F^{-1}(\alpha)} - \frac{1}{2}
+ \BigO{\varepsilon}\right).
$$

Summing the lower bounds above for $N_{\mathcal{L}_\rho}$, $N_{\mathcal{A}}$, $N_{\Gamma_\alpha, \varepsilon}$, we get that these account for at least $(k+o(k))\left(\frac{1}{12} + \BigO{\varepsilon}\right)$ zeros.
Since we can take $\varepsilon$ arbitrarily small (and for each $\varepsilon$ take $k$ sufficiently large), it follows that this accounts for $100\%$ of the zeros of $P_{k,m}$.

\end{proof}

 % non-real zeros

\bibliographystyle{plain}
\bibliography{my_bib}

\end{document}